\documentclass[11pt,reqno,oneside]{amsart}
\usepackage{hyperref}
\usepackage{geometry}
\usepackage{graphicx}
\usepackage{amssymb}
\usepackage{pgfplots}
\usepackage{amsmath, mathrsfs, stmaryrd, color,slashed}
\newtheorem{theorem}{Theorem}

\usepackage[title]{appendix}
\newtheorem{remark}[theorem]{Remark}
\newtheorem{example}[theorem]{Example}
\newtheorem{definition}[theorem]{Definition}

\newtheorem{corollary}[theorem]{Corollary}
\newtheorem{lemma}[theorem]{Lemma}
 
\newtheorem{claim}[theorem]{Claim}
\newcommand{\R}{\mathbb R}

\numberwithin{theorem}{section}
\numberwithin{equation}{section}

\begin{document} 

\title {Inverse Mean Curvature Flow with Singularities}
\author{Beomjun Choi}
\address{Beomjun Choi\\
Department of Mathematics\\
Columbia University, USA}
\email{bc2491@columbia.edu}
\author{Pei-Ken Hung}
\address{Pei-Ken Hung\\
Department of Mathematics\\
Massachusetts Institute of Technology, USA}
\email{pkhung@mit.edu}

\begin{abstract}
This paper concerns the inverse mean curvature flow of complete convex hypersurfaces which are Lipschitz in general. After defining a weak solution, we study the evolution of singularities by looking at the blow-up tangent cone around each singular point.  We prove the cone also evolves by the inverse mean curvature flow and each singularity is removed when the evolving cone becomes flat.  As a result, we derive the exact waiting time for a weak solution to be a smooth solution. In particular, a necessary and sufficient condition for an existence of smooth classical solution is given.  
 \end{abstract}
\maketitle

\section{Introduction}
A one-parameter family of smooth embedded hypersurfaces $F:M^n\times [0,T] \to \mathbb{R}^{n+1}$ is a classical solution of the {\em inverse mean curvature flow} (IMCF) if \[\frac{\partial\,}{\partial t} F(p,t) = \frac{1}{H(p,t)} \nu (p,t)\quad \text{for all} \quad (p,t) \in M\times[0,T] \] where $H(p,t)>0$ and $\nu(p,t)$ are the mean curvature and the outward unit normal vector of $M_t := F(M\times \{t\})$ at the point $F(p,t)$.

The inverse mean curvature flow has been studied extensively as an important example of expanding curvature flows last decades \cite{Ge2,Ur,Sm,HI,Ge}. After the IMCF was used to give a proof of the Riemannian Penrose inequality by Huisken and Ilmanen \cite{HI1,HI2}, there have been a number of application of it in showing geometric inequalities in various situations \cite{makowski2013rigidity,de2016alexandrov,ge2014hyperbolic,guan2007quermassintegral,BHW,lee2015penrose}. 

\medskip
``What happens if the initial hypersurface is a cube in $\mathbb{R}^3$?" This was the original question we had. 
Most of previous research considered the conditions on initial hypersurface which produce a smooth solution for $t>0$. Even in the weak formulation of Huisken and Ilmanen  \cite{HI2,HI}, the mean curvature of initial hypersurface is assumed to be uniformly bounded in a weak distribution sense. If singular initial hypersurfaces are considered, however, singularities may not be removed instantaneously. The simplest example with this phenomenon is a round cone solution\[M_t:= \{(x',x_{n+1})\in\mathbb{R}^{n+1}\,:\, x_{n+1}=\tan (\theta(t)) |x'|\}.\]
When $\theta'(t) = -\frac{\cot(\theta(t))}{n-1}$, $M_t$ is a smooth solution of the IMCF outside the origin. In this paper, we take a slightly different approach to the problem and set our goal to understand the behavior of singularities under the flow.   

The previous example shows one difference between the IMCF and the {\em mean curvature flow} where we have an existence of smooth solution for Lipschitz hypersurfaces. In a PDE point of view, it is closely related with the type of diffusion equation the curvature is following under the IMCF. For the IMCF in $\mathbb{R}^{n+1}$, the mean curvature evolves by an ultra-fast diffusion\[\partial_tH =\nabla \cdot (H^{-2} \nabla H) - \frac{|A|^2}{H}.\] Apparently, the diffusion coefficient becomes zero when $H$ is infinite, which prevents the singularities from being removed.

Let's discuss our result. We restrict our attention to the class of complete convex initial hypersurfaces. Note that on a convex hypersurface every singular point has a unique convex tangent cone which is obtained by taking a blow-up. In the meantime, IMCF has a scaling property that if $M_t$ is a solution, then $\tilde M_t =\lambda M_t$ is again a solution; i.e., it does not scale in time variable.  This suggests that the blow-up tangent cone at a singular point may also evolve by IMCF in the same time scale. After defining a weak solution in Definition \ref{weak solution}, our main results, Theorem \ref{thm-euc} and Theorem \ref{thm-sph}, prove this assertion for the weak flows in $\mathbb{R}^{n+2}$ and $\mathbb{S}^{n+1}$, respectively. 

There is an inevitable reason why the weak flows in $\mathbb{R}^{n+2}$ and $\mathbb{S}^{n+1}$ need to be  simultaneously considered. We have a correspondence between hypersurfaces in $\mathbb{S}^{n+1}$ and conical hypersurfaces in $\mathbb{R}^{n+2}$: a hypersurface $\Gamma^n $ in $ \mathbb{S}^{n+1}$ generates a cone $$\mathcal{C} \Gamma:= \{r x\in\mathbb{R}^{n+2} \,:\, r>0 ,\, x\in \Gamma\}, $$ and conversely a conical hypersurface $C\subset \mathbb{R}^{n+2}$ is generated by its link $\Gamma:=C\cap \mathbb{S}^{n+1}$. Moreover, if $\Gamma^n_t$ in $ \mathbb{S}^{n+1}$ is a classical solution of IMCF, then the cone $\mathcal{C} \Gamma_t$  in $\mathbb{R}^{n+2}$ is a solution of IMCF which is smooth except the origin. If we take a blow-up tangent cone at a singular point in $\Sigma^{n+1}\subset \mathbb{R}^{n+2}$, the cone can also be singular as its link in $\mathbb{S}^{n+1}$ can be an arbitrary convex hypersurface. Hence, if we intend to describe singular solutions in $\mathbb{R}^{n+2}$, it is necessary to work with singular solutions in $\mathbb{S}^{n+1}$ as well.

The result of Gerhardt \cite{Ge} and Makowski-Scheuer  \cite{makowski2013rigidity} which is summarized in our Theorem \ref{smooth} shows every strictly convex smooth solution in $\mathbb{S}^{n+1}$ converges to an equator in a finite time. Since the tangent cone at a singularity also evolves by the weak IMCF, this suggests that the tangent cone becomes flat and the singularity at a point will be removed in a finite time.  When all singularities are removed in this way, one should expect the solution becomes smooth afterward. We prove this result in Corollary \ref{cor-smoothingtime}. In fact, this exact waiting time can be written in terms of the density function on initial hypersurface and the same argument gives a sufficient and necessary condition for an existence of smooth classical solution for $t>0$. These are given Remark \ref{finremark}. 
\section{Preliminaries and Definitions}
Every hypersurface appear from here is complete without boundary unless otherwise mentioned. 
In  $\mathbb{S}^{n+1}$, there are several different notions for a hypersurface $\Gamma^n$ being convex. We will use the following definition:

\begin{definition}\label{def-convexity}
A compact Lipschitz hypersurface $\Gamma^n\subset \mathbb{S}^{n+1}$ is convex if the corresponding cone $\mathcal{C}\Gamma$:=$\left\{rx \,:\,r\in [0,\infty),\ x\in \Gamma^n\subset\mathbb{R}^{n+2}  \right\}\subset\mathbb{R}^{n+2}$ is convex.
\end{definition}
\begin{remark}
From the above definition, a convex hypersurface $\Gamma^n\subset \mathbb{S}^{n+1}$ is contained in some hemisphere because of the supporting hyperplane of $\mathcal{C}\Gamma$ at the origin.
\end{remark}
In the rest of this paper, when we mention $\mathbb{S}^{n+1}$ and $\mathbb{R}^{n+2}$, $n\ge1$ is assumed unless it is stated otherwise.

\begin{remark}\label{antipodal_pts}
If $\Gamma^n\subset \mathbb{S}^{n+1}$ is convex and contains no antipodal points, then $\Gamma^n$ is compactly contained in some open hemisphere. For this proof, see for instance Lemma 3.8 \cite{makowski2013rigidity}.
\end{remark}

We will use the notion of strict convexity only for a smooth hypersurfaces and it means the hypersurface has strictly positive second fundamental form. For hypersurfaces in the sphere, we further require them to be compactly contained in some open hemisphere. We define $\mathbb{S}^{n+1}_+:=\left\{x\in\mathbb{S}^{n+1}\Big|\ \left\langle x,e_{n+2}\right\rangle>0 \right\}$.
\\
%

Expanding curvature flows for strictly convex hypersurfaces in $\mathbb{S}^{n+1}$ has been studied by C. Gerhardt \cite{Ge} and M. Makowski and J. Scheuer \cite{makowski2013rigidity}. The following is a special case of  Theorem 1.4 in \cite{makowski2013rigidity} when the \textit{inverse mean curvature flow} is considered. $T_{\Gamma_0}$ is not specified in the paper, but this time could be obtained explicitly by the exponential growth of the area with respect to time under the {\em inverse mean curvature flow}; i.e., $|\Gamma_t|=e^t |\Gamma_0|$.
\begin{theorem}\label{smooth}
Suppose $\Gamma^n_0\subset \mathbb{S}^{n+1}$ is a smooth, strictly convex hypersurface. Then the unique smooth, strictly convex solution of inverse mean curvature flow $\Gamma_t$ starting from $\Gamma_0$ exists for $t\in [0,T_{\Gamma_0})$, where $T_{\Gamma_0}=\ln|\mathbb{S}^n|-\ln|\Gamma_0|$. Furthermore, $\Gamma_t$ converges to an equator in $C^{1,\beta}$  as $t\to T_{\Gamma_0}$.
\end{theorem}
For a complete convex hypersurface $\Sigma^{n+1}\subset \mathbb{R}^{n+2}$, we denote by $\hat{\Sigma}$ the closed region bounded by $\Sigma$. Similarly,  for a complete convex hypersurface $\Gamma^{n}\subset \mathbb{S}^{n+1}$ we denote by $ \hat\Gamma $ the closed region bounded by $\Gamma$ in $\mathbb{S}^{n+1}$. Since $\Gamma^n$ is in some hemisphere, there is no ambiguity in this definition unless $\Gamma$ is an equator. 

We need to define our notion of a weak solution. Since it is expected that singularities will persist, we can't describe this in terms of the classical solution. Also, the level set based weak formulation of Huiksen and Ilmanen \cite{HI2} would not be appropriate as it assumes bounded (weak) mean curvature of level sets. We just define our solution as a limit of a smooth approximation from inside. In the following definition, note $\Sigma_0^{n+1}$ can be non-compact.

\begin{definition}\label{weak solution}
For a complete convex hypersurface $\Sigma^{n+1}_0\subset \mathbb{R}^{n+2}$, let $\{\Sigma_{0,\epsilon}\}_{\epsilon\in(0,\epsilon_0)}$ be a family of smooth, strictly convex, compact hypersurfaces in $\mathbb{R}^{n+2}$ with the following properties:
\begin{align*}
\hat{\Sigma}_{0,\epsilon_2}\subset\subset\hat{\Sigma}_{0,\epsilon_1}&\ \textup{if}\ \epsilon_2>\epsilon_1,\quad\text{ and }\quad
\overline{\bigcup_{\epsilon>0}\hat{\Sigma}_{0,\epsilon}}=\hat{\Sigma}_0.
\end{align*} 
Let $\Sigma_{t,\epsilon}$ be the \textit{inverse mean curvature flow} starting from $\Sigma_{0,\epsilon}$. We define the (weak) solution of the flow $\Sigma_t$, for $t\ge0$, starting from $\Sigma_0$ by 
\begin{align}
\hat{\Sigma}_t&:=\overline{\bigcup_{\epsilon>0}\hat{\Sigma}_{t,\epsilon}},\quad\text{ and }\quad \Sigma_t:=\partial \hat{\Sigma}_t.
\end{align}
Similarly, for a complete convex compact hypersurface $\Gamma^n_0 \subset \mathbb{S}^{n+1}$, suppose $|\Gamma^n_0|<|\mathbb{S}^n|$ and hence $ T_{\Gamma_0}:=\ln|\mathbb{S}^n|-\ln |\Gamma_0|>0$. Let $\{\Gamma_{0,\epsilon}\}_{\epsilon\in(0,\epsilon_0)}$ be a family of smooth, strictly convex, compact hypersurfaces in $\mathbb{S}^{n+1}$ with the following properties:
\begin{align*}
\hat{\Gamma}_{0,\epsilon_2}\subset\subset\hat{\Gamma}_{0,\epsilon_1}&\ \textup{if}\ \epsilon_2>\epsilon_1,\quad\text{ and }\quad
\overline{\bigcup_{\epsilon>0}\hat{\Gamma}_{0,\epsilon}}=\hat{\Gamma}_0.
\end{align*} 
Let $\Gamma_{t,\epsilon}$ be the \textit{inverse mean curvature flow} starting from $\Gamma_{0,\epsilon}$. We define (weak) solution of the flow $\Gamma_t$, for $t\in[0,T_{\Gamma_0} )  $, starting from $\Gamma_0$ by 
\begin{align}
\hat{\Gamma}_t&:=\overline{\bigcup_{\epsilon>0}\hat{\Gamma}_{t,\epsilon}},\quad\text{ and }\quad \Gamma_t:=\partial \hat{\Gamma}_t.
\end{align}

\end{definition}
\begin{remark} When initial hypersurface is compact, an approximating family $\Sigma_{0,\epsilon}$ and $\Gamma_{0,\epsilon}$ in the definition could be obtained by running the mean curvature flow from $\Sigma_0$ and $\Gamma_0$, respectively. When $\Sigma_0$ in $\mathbb{R}^{n+2}$ is non-compact, we may choose $\Sigma_{0,\epsilon}$ to be the time $t=\epsilon$ slice of the mean curvature flow running from $\partial ( \hat\Sigma_0 \cap B_{\epsilon^{-1}}(0))$. By using the avoidance principle between smooth compact solutions of the IMCF, it is not hard to check the weak solution is independent of  approximations and hence unique.  

In the definition of $\Sigma_t$, $\Sigma_{t,\epsilon}$ are convex since convexity is preserved under the flow by Corollary A.4 in \cite{CD}. Therefore $\Sigma_t$ is convex. Moreover, when $\Sigma_0$ is compact, $\Sigma_t$ is non-empty compact hypersurface for all $t>0$ and the exponential in time area growth of classical solution imply $$|\Sigma_t|=\lim_{\epsilon\to 0} |\Sigma_{t,\epsilon}|=\lim _{\epsilon\to 0}e^t |\Sigma_{0,\epsilon}|= e^t |\Sigma_0|.$$ When $\Sigma_0$ is non-compact, it may happen that $\hat \Sigma_t=\mathbb{R}^{n+2}$ and hence $\Sigma_t$ becomes empty in finite time. We will discuss non-compact case  with complete details in Remark \ref{non-compact}.

In the definition of $\Gamma_t$, since $\Gamma_{0,\epsilon}$ has a smaller area than $\Gamma_0$, $\Gamma_{t,\epsilon}$ exists at least for $t\in[0,\ln|\mathbb{S}^n|-\ln |\Gamma_0| )$ by Theorem \ref{smooth}. Therefore, we can take the limit $\epsilon\to 0$ for $t<\ln|\mathbb{S}^n|-\ln |\Gamma_0|$. 
Note also that $ \Gamma_t$ is convex in the sense of Definition \ref{def-convexity}.  Moreover,  $$|\Gamma_t|=\lim_{\epsilon\to 0} |\Gamma_{t,\epsilon}|=\lim _{\epsilon\to 0}e^t |\Gamma_{0,\epsilon}|= e^t |\Gamma_0|.$$

\end{remark}
 
We will need the following apriori estimate of $H^{-1}$ weighted by $|x|^{-1}$. This is main theorem in \cite{CD} and will be used in an important way later.
\begin{theorem}[Theorem in \cite{CD}] \label{thm-CD}
For $n\ge2$, let $F:M^n\times [0,T] \to \R^{n+1}$ be a smooth convex compact solution of IMCF and assume for $t\in[0,T]$ there is  $\theta_1 \in (0,\pi/2)$ and a unit vector $\omega \in \mathbb{R}^{n+1}$ so that \[ \langle F , \omega \rangle \ge  {\sin\theta_1}\, |F|.\]
Then \[\frac{1}{H|F|}\le C\left( 1+ {t^{-1/2}}\right) \quad\text{on}\quad M\times[0,T]\] by some $C=C(\theta_1)>0$.
\end{theorem}

\section{Evolution of Singularity}

\begin{lemma}\label{main lemma}
Let $\Sigma_0^{n+1}\subset\mathbb{R}^{n+2}$ be a complete convex hypersurface which contains the origin. Moreover, suppose the tangent cone at the origin $T_0\Sigma_0$ has a smooth, strictly convex link $$\Gamma_0^n:= T_0\Sigma_0\cap \mathbb{S}^{n+1}\subset \mathbb{S}^{n+1},$$ and $B^{n+2}_1(0)\cap\Sigma_0=B^{n+2}_1(0)\cap T_0\Sigma_0$; i.e., $\Sigma_0$ is conical in the ball of radius $1$ centered at the origin.  Let  $\Sigma_t$ and $\Gamma_t$ be the IMCF starting from $\Sigma_0$ in $\mathbb{R}^{n+2}$ and $\Gamma_0$ in $\mathbb{S}^{n+1}$ respectively. Then $0\in \Sigma_t$ for $t\in[0,T_{\Gamma_0})$ and $T_0 \Sigma_t$ coincides with $\mathcal{C}\Gamma_t$.
\end{lemma}
\begin{proof}
For a fixed smooth approximation of $\Sigma_0$, namely $\Sigma_{0,\epsilon}$, we have $\hat{\Sigma}_{0,\epsilon}\subset\hat{\Sigma}_{0}\subset \mathcal{C}\hat{\Gamma}_0$. $\hat{\Sigma}_{t,\epsilon}\subset \mathcal{C}\hat{\Gamma}_t$ for any $0\le t<T_\Gamma$ by the avoidance principle between smooth solutions.  Thus $0\in \Sigma_t$ for $t\in[0,T_\Gamma)$ by taking limit $\epsilon\to 0$. Since IMCF preserves convexity, $\Sigma_t$ is convex and hence $T_0\Sigma_t$ is a cone. Let us denote this by $\mathcal{C}\Gamma'_t$. We have $\hat{\Gamma}_0\subset\hat{\Gamma}'_t\subset \hat{\Gamma}_t$ from $\hat{\Sigma}_0\subset\hat{\Sigma}_t\subset \mathcal{C}\hat{\Gamma}_t$. To show $\mathcal{C}\Gamma'_t=\mathcal{C}\Gamma_t$, by the uniqueness of smooth solution and $\Gamma'_0=\Gamma_0$, it suffices to prove $\Gamma'_t$ is a smooth solution of IMCF in $\mathbb{S}^{n+1}$ for $t\in[0,T_{\Gamma_0})$. \\\\
We will show $\Gamma'_t$ is a smooth IMCF on $t\in[0,T_{\Gamma_0}-\delta)$ for all small $\delta>0$. After a rotation, we may assume that $e_{n+2} \in \text{int}\hat{\Gamma}_{0}$. By Theorem \ref{smooth}, there is $v\in \mathbb{S}^{n+1}$ such that $\Gamma_t \subset H(v):=\{v'\in \mathbb{S}^{n+1} \,|\, \langle v,v'\rangle>0 \} $ and $\Gamma_t$ converges to an equator $S(v):=\{ v' \in \mathbb{S}^{n+1} \, | \, \langle v, v' \rangle =0\}$ as $t\to T_{\Gamma_0}$. Since $\Gamma_t$ is a smooth, strictly convex solution, \[\epsilon_0 := \inf \left\{ \langle v',  v \rangle \big| v'\in{\Gamma_{T-\delta}}\right\}>0.\]
Since weak solution $\Sigma_t$ is independent of  approximation, we can choose $\Sigma_{0,\epsilon}$ so that it converges to $\Sigma_0$ locally smoothly on $B_1(0)\setminus\{0\}$. 
 Let $\{\lambda_j\} $ be a sequences of numbers which increases to infinite. It is possible to pick a decreasing sequence $\epsilon_k \to 0$ so that $\lambda_j \hat\Sigma_{0,\epsilon_k}$ locally uniformly converges to $\mathcal{C}\Gamma'_{0}$ as $j\to \infty$ and $k\geq j$. Let us denote a strip $$D_{\epsilon_0}:=\{x \in \mathbb{R}^{n+2} \, |\, \epsilon_0<\langle v,x \rangle <1\}.$$ We want to show for $t\in [0,T-\delta]$, $\displaystyle\lim_{j\to\infty}\lim_{k\to\infty}\lambda_j\Sigma_{t,\epsilon_k}\cap D_{\epsilon_0}$ converges smoothly to an (incomplete) IMCF.
From now on we assume $\lambda_i\ge3$ by assuming $i\ge i_0$. Note \begin{align*}\lambda_j\Sigma_0 \cap \{0 < x_{n+2} < 2\}= \mathcal{C}\Gamma_0 \cap\{ 0< x_{n+2} < 2\}.\end{align*} 
Due to locally smooth convergence of $\Sigma_{\epsilon,0}$, we may choose further smaller deceasing seqeunce $\epsilon_k$ and fix it so that for any $x\in D_{\epsilon_0}$ and any $k\geq j$
\begin{align*}
\sup_{B_{\epsilon_0/2}(x)\cap \lambda_j\Sigma_{\epsilon_k,0}} H \leq 2\sup_{\{|x|>{\epsilon_0}/2\}\cap \mathcal{C}\Gamma}H= C{\epsilon_0}^{-1}<\infty.
\end{align*}
Let us denote  $\Sigma_{jk,t}:=\lambda_{j} \Sigma_{t,\epsilon_k}$, which is again a smooth solution of IMCF by the scaling property. 
Due to a local estimate of $H$ in Proposition 2.11 \cite{DH}, we have a uniform upper bound of $H$ for $\Sigma_{jk,t}$ in $D_{\epsilon_0}$ for $t\in[0,T_{\Gamma_0})$. Next, by the avoidance principle between $\Sigma_{jk,t}$ and $\mathcal{C} \Gamma_t$, \[\langle x, v \rangle \ge \epsilon_0 |x|\quad\text{for all }x\in \Sigma_{jk,t} \text{ with } t\in[0,T_{\Gamma_0}-\delta]. \] Hence by Theorem \ref{thm-CD}, during this time interval, there is a $ C=C(\epsilon_0)$ such that 
\begin{align*}
\frac{1}{H|x|}\leq C\left(1+t^{-1/2} \right)\quad \text{for }x\in \Sigma_{jk,t} \text{ with } t\in[0,T_{\Gamma_0}-\delta].
\end{align*}

Now we express $\Sigma_{jk,t}\cap D_{\epsilon_0}$ as a graph over a tilted cylinder as the following. For any $s\in (\epsilon_0,1)$, $\Sigma_{jk,t}\cap \{x\,|\,\langle v,x\rangle=s\}$ is a convex hypersurface in $\{x\,|\,\langle v,x\rangle=s\} \approx \mathbb{R}^{n+1}$.  Since $e_{n+2} \in \text{int}\hat{ \Gamma}_0$, for large $j$ and $k\geq j$, this convex hypersurface contatins the point $\frac{s}{\langle v, e_{n+2} \rangle} e_{n+2}$. Therefore, we may express $\Sigma_{jk,t}\cap \{\langle v,x\rangle=s\}$ in the polar coordinate centered at this point.
\begin{align*}
\{r'=r_{jk}'(t;s,\theta):\theta\in \mathbb{S}^{n}\} :=\Sigma_{jk,t}\cap \{\langle v,x\rangle=s\}.
\end{align*}
$r_{jk}'(t;.)$ is uniformly bounded from above and from below because $\hat{\Sigma}_{jk,0}\subset \hat{\Sigma}_{jk,t}\subset \mathcal{C}\hat\Gamma_t$. The spatial $C^1$ bound of $r'_{jk}$ follows directly from the convexity and previous $C^0$ bounds. The evolution of this scalar function $r'(t;s,\theta)$ is \[\partial_t r' = \frac{ \langle \frac{\partial \,}{\partial r'} , \nu \rangle}{H}= F(D^2 r', D r'  ,r ,s,\theta)\] and it is uniformly parabolic if $H$, $H^{-1}$ and spatial $C^1$ norm of $r'$ are bounded. By this assertion and estimates above, as long as $k\geq j$, $r'_{jk}(t;s,\theta)$ are uniformly bounded and satisfy a uniform parabolic equation in $(\delta,T_\Gamma-\delta)\times (\epsilon,1)\times \mathbb{S}^n$. From the parabolic regularity theory, $\displaystyle \lim_{j\to\infty}\lim_{k\to\infty}r_{jk}'(t;s,\theta)$ converges locally smoothly to $r_\infty'(t;s,\theta)$. Thus $\mathcal{C}\Gamma'_t\cap D_\epsilon$ satisfies IMCF and so does $\Gamma'_t$.

\end{proof}

\begin{theorem} \label{thm-euc}
Let $\Sigma_0^{n+1}\subset\mathbb{R}^{n+2}$ be a complete convex hypersurface which contains the origin so that the tangent cone at the origin $T_0\Sigma_0$ has a convex link $$\Gamma_0^n:= T_0\Sigma_0\cap \mathbb{S}^{n+1}\subset \mathbb{S}^{n+1}.$$ 
Let $T_{\Gamma_0}:= \ln |\mathbb{S}^{n}|-\ln |\Gamma^n_0|$. Then $0\in \Sigma_t$ for $t\in[0,T_{\Gamma_0}]$ and $0\in \text{int }\hat\Sigma_t$ for $t\in(T_{\Gamma_0},\infty)$. Suppose $T_{\Gamma_0}>0$ and $\Gamma_t\subset \mathbb{S}^{n+1}$ be the IMCF starting from $\Gamma_0$.  Then $T_0 \Sigma_t$ coincides with $\mathcal{C}\Gamma_t$  for $t\in[0,T_{\Gamma_0})$.
\end{theorem}

\begin{proof}
We again denote by $\Gamma'_t:=T_0\Sigma_t\cap \mathbb{S}^{n+1}$ {provided $0\in \Sigma_t$}. Let $\{\Gamma_{0,\epsilon}\}_{\epsilon\in (0,\epsilon_0)}$ be a family of smooth, strictly convex hypersurfaces which approximate $\Gamma_0$ as in Definition \ref{weak solution}. For any other smooth, strictly convex hypersurface $\hat{\Sigma}'_0\subset\subset \hat{\Sigma}_0$, we have $\hat{\Sigma}'_0\subset\subset \mathcal{C}\Gamma_{0,\epsilon}$ as $\epsilon$ is small enough. Therefore from the avoidance principle between smooth solutions, we obtain that  $0\in \Sigma_t$ and $\hat{\Gamma}'_t\subset \hat{\Gamma}_t$ for $t\in[0,T_{\Gamma_0}]$. To get the other direction of inclusion, for each $\epsilon\in (0,\epsilon_0)$, let $\Sigma_{0,\epsilon}$ be a hypersurface in $\mathbb{R}^{n+2}$ which satisfies the assumptions in Lemma \ref{main lemma} with the link $\Gamma_{0,\epsilon}$. Then there exits a small number $a_\epsilon>0$ such that $a_\epsilon \hat{\Sigma}_{0,\epsilon}\subset \hat{\Sigma}_0$. By the avoidance principle and Lemma \ref{main lemma}, we obtain $\hat{\Gamma}_{t,\epsilon }\subset \hat{\Gamma}'_t$. Then the result follows by taking $\epsilon\to 0$.\\

To show that $0\in \text{int}\hat\Sigma_t$ for $t\in(T_{\Gamma_0},\infty)$, we first assume $\Gamma_0$ is smooth and strictly convex. By Theorem \ref{smooth}, $\Gamma_t$ converges in $C^{1,\beta}$ to an equator as $t\to T_{\Gamma_0}$. Thus for any $\delta>0$, there exists a small number $a_\delta>0$ such that $\hat{\Sigma}_{T_{\Gamma_0}}$ contains an $a_\delta$-radius ball with its center located in distance $a_\delta\exp\left({\delta}/(n+1)\right)$ from the origin. By running the IMCF from this sphere, we deduce $0\in\text{int}\hat{\Sigma}_t$ for any $t> T_{\Gamma_0}+\delta$. By taking $\delta\to 0$ we conclude the case in which $\Gamma_0$ is smooth and strictly convex. The case of general $\Gamma_0$ can be proved by a smooth approximation.
\end{proof}

For a point $p\in \Gamma^n\subset \mathbb{S}^{n+1}$, if $\Gamma^n$ is convex, there is a unique n-dimensional tangent cone at $p$ which is denoted by $T_p\Gamma^n$. More precisely, $\mathcal{C}\Gamma \subset \mathbb{R}^{n+2}$ is convex hypersurface in $\mathbb{R}^{n+2}$ and we define using tangent cone of $\mathcal{C}\Gamma$ at $p$ by \[T_p\Gamma^n := T_p\mathcal{C}\Gamma \cap \{ x\in \R^{n+2} \, |\, \langle x-p, p \rangle=0\} .\]
\begin{theorem}\label{thm-sph}
Let $\Gamma_0^n\subset\mathbb{S}^{n+1}$ be a closed convex hypersurface so that $p\in \Gamma_0$ has the tangent cone $T_p\Gamma_0$ with the convex link in $ \{ x\in \R^{n+2} \, |\, \langle x-p, p \rangle=0\text{ and }||x-p||=1\}:=\mathbb{S}^n_p\approx\mathbb{S}^{n}$ which is denoted by $$\Theta_0^{n-1}:= T_p\Gamma^n_0\cap \mathbb{S}^n_p.$$ 
Then $T_{\Theta_0}= \ln |\mathbb{S}^{n-1}|_{n-1}-\ln |\Theta^{n-1}_0|_{n-1}\le T_{\Gamma_0}= \ln |\mathbb{S}^n|_n-\ln|\Gamma^n_0|_n $. {The equality holds if and only if $T_p\mathcal{C}\Gamma_0 \cap \mathbb{S}^{n+1} = \Gamma_0^n$ provided $\mathcal{C}\Theta_0\neq \mathbb{R}^n$}. Moreover, $p\in \Gamma_t$ for $t\in[0,T_{\Theta_0}]$ and $p\in \text{int}\hat\Gamma_t$ for $t\in(T_{\Theta_0},T_{\Gamma_0})$. Suppose $T_{\Theta_0}>0$ and $\Theta_t$ be the IMCF starting from $\Theta_0$ in $\mathbb{S}^n_p$.  Then $T_p\Gamma^n_t\cap \mathbb{S}^n_p$ coincides with $\Theta_t$  for $t\in[0,T_{\Theta_0})$.\end{theorem}
The proof will be very similar to the proof of Theorem \ref{thm-euc}, but we need lemmas to explain relations between $\Theta_0$ and $\Gamma_0$.
\begin{lemma}\label{area equality}
Let $n\geq k\geq 0$ and $\Theta^k\subset\mathbb{S}^{k+1}\subset\mathbb{R}^{k+2}$ be a hypersurface. We denote by $\Gamma^n:=\left(\mathbb{R}^{n-k}\times\mathcal{C}\Theta^k\right)\cap \mathbb{S}^{n+1}$ (or equivalently $\mathcal{C}\Gamma=\mathbb{R}^{n-k}\times\mathcal{C}\Theta$.) Then
\begin{align*}	
\frac{|\Theta^k|_k}{|\mathbb{S}^k|_k}=\frac{|\Gamma^n|_n}{|\mathbb{S}^n|_n}.
\end{align*}
\end{lemma}

\begin{lemma}\label{lem-area comparison}
Under the same assumptions in Theorem \ref{thm-sph}, we have
\begin{align}\label{area comparison}
\frac{|\Gamma_0|_n}{|\mathbb{S}^n|_n}\leq\frac{|\Theta_0|_{n-1}}{|\mathbb{S}^{n-1}|_{n-1}}.
\end{align}
Moreover, if $\mathcal{C}\Theta_0 \neq \mathbb{R}^{n}$, then the equality holds if and only if $\mathcal{C}\Gamma_0\cong\mathbb{R}\times \mathcal{C}\Theta_0$.
\end{lemma}
\begin{proof}[Proof of Lemma \ref{lem-area comparison}]
Without loss of generality, we can assume that $p=e_1$. Let $\Gamma_0':=\left(\mathbb{R}\times \mathcal{C}\Theta_0\right)\cap \mathbb{S}^{n+1}$. We have $\mathcal{C}\hat\Gamma_0\subset \mathcal{C}\hat\Gamma'_0$ from the convexity of $\Gamma_0$. Since convex hypersurfaces are outer area minimizing in a hemisphere, \eqref{area comparison} follows from {Lemma \ref{area equality}}. It's obvious that $\mathcal{C}\Gamma_0=\mathbb{R}\times \mathcal{C}\Theta_0$ implies the equality in \eqref{area comparison}. From now on, we assume $\mathcal{C}\Gamma_0\neq \mathbb{R}\times \mathcal{C}\Theta_0$. If $-e_1\in \Gamma_0$ then $\mathcal{C}\Gamma_0= \mathbb{R}\times \mathcal{C}\Theta_0$ from the convexity of $\Gamma_0$. Therefore, we have $-e_1 \notin \Gamma_0$.\\

If we first assume that $\Theta_0$ is compactly contained in some hemisphere, say in $\mathbb{S}^n_+$, then there exists a small positive number $\epsilon_0>0$ such that $\left\langle x,v(\epsilon_0) \right\rangle\geq 0$ for all $x\in \Gamma_0$, where $v(\epsilon_0):=\sin\epsilon_0\, e_1+\cos \epsilon_0\,e_{n+2}$. Define
\begin{align*}
\hat{\Gamma}''_0:=&\hat{\Gamma}'_0\cap \left\{x\in\mathbb{S}^{n+1}\Big|\ \left\langle x,v(\epsilon_0) \right\rangle\geq 0 \right\}\ \text{and}\ \Gamma''_0:=\partial \hat{\Gamma}''_0.
\end{align*}
We note that $\hat{\Gamma}_0\subset \hat{\Gamma}''_0$ and hence $|\Gamma_0|_n\leq |\Gamma''_0|_n$ since convex hypersurfaces are outer area minimizing in a hemisphere. Denote by $\mathcal{D}:=\overline{\hat{\Gamma}'_0\setminus \hat{\Gamma}''_0} $. The boundary of $\mathcal{D}$ consists two parts $\partial_+\mathcal{D}$ and $\partial_-\mathcal{D}$, where
\begin{align*}
\partial_+\mathcal{D}:=&\Gamma'_0\setminus \left\{x\in\mathbb{S}^{n+1}\Big|\ \left\langle x,v(\epsilon_0) \right\rangle> 0 \right\},\\
\partial_-\mathcal{D}:=&\hat{\Gamma}'_0\cap\left\{x\in\mathbb{S}^{n+1}\Big|\ \left\langle x,v(\epsilon_0) \right\rangle= 0 \right\}.
\end{align*} 
Moreover, 
\begin{align*}
|\Gamma'_0|_n-|\Gamma''_0|_n=|\partial_+\mathcal{D}|_n-|\partial_-\mathcal{D}|_n.
\end{align*}
In the region $\mathcal{D}$, there is a minimal foliation \[\mathcal{D}\cap\left\{x\in\mathbb{S}^{n+1}\Big|\ \left\langle x,v(\epsilon) \right\rangle= 0 \right\}_{\epsilon\in (0,\epsilon_0)}.\] Note that $\partial_- \mathcal{D}$ belongs to this minimal foliation and $\partial_+\mathcal{D}$ doesn't. Therefore by applying the divergence theorem to the normal vector of the foliation in the region $\mathcal{D}$, we obtain $|\Gamma_0|_n\leq |\Gamma''_0|_n<|\Gamma'_0|_n$.\\

For general $\Theta_0$, from Remark \ref{antipodal_pts}, we have $\mathcal{C}\Theta_0=\mathbb{R}^{k}\times \mathcal{C}\Theta'_0$, for some $\Theta'_0$ compactly contained in $\mathbb{S}^{n-k}_+$. By the assumption $\mathcal{C}\Theta_0\neq \mathbb{R}^n$, $k\leq n-1$. Denote by $\pi_k:\mathbb{R}^{n+2}\longrightarrow \mathbb{R}^{n-k+2}$ the projection which suppresses the second to the $k+1$-th coordinate. Then by applying the above argument to $\pi_k \left(\mathcal{C}\Gamma_0\right)$, we obtain
\begin{align*}
\frac{|\Gamma_0|_n}{|\mathbb{S}^n|_n}\leq &\frac{|\pi_k \left(\mathcal{C}\Gamma_0\right)\cap \mathbb{S}^{n+1-k}|_{n-k}}{|\mathbb{S}^{n-k}|_{n-k}}\\
< &\frac{|\Theta'_0|_{n-k-1}}{|\mathbb{S}^{n-k-1}|_{n-k-1}}\\
= &\frac{|\Theta_0|_{n-1}}{|\mathbb{S}^{n-1}|_{n-1}}.
\end{align*}
\end{proof}
In previous lemma, we didn't figure out the equality case of \eqref{area comparison} when $\mathcal {C}\Theta_0=\mathbb{R}^{n}$. In this case we get \[\frac{|\Gamma_0|_n}{|\mathbb{S}^n|_n}=\frac{|\Theta_0|_{n-1}}{|\mathbb{S}^{n-1}|_{n-1}}= \frac{|\mathbb{S}^{n-1}|_{n-1}}{|\mathbb{S}^{n-1}|_{n-1}}=1.\] 
This is of a separate interest to classify such $\Gamma_0$ and the following proves this: 
\begin{lemma} \label{wedge}If $|\Gamma_0 |_n= |\mathbb{S}^n|_n$, then up to an isometry of $\mathbb{S}^{n+1}$, $\Gamma_0$ is either an equator or an wedge \[W_{\theta_0}=\mathbb{S}^{n+1}\cap\bigl(\{ (r\sin\theta, r\cos\theta)\,:\,  \theta\in \{0,\theta_0\}, \text{ and }r>0 \}\times \mathbb{R}^n\bigr)\] for some $\theta_0\in(0,\pi)$.
\begin{proof}[Proof of Lemma \ref{wedge}]
 There is a point $q\in\Gamma_0$ such that the link generated at $q$, say $\bar \Theta_0= T_q \Gamma_0^n \cap \mathbb{S}^n_q$, satisfies $\mathcal{C}\Gamma_0\cong \mathbb{R} \times \mathcal{C}\bar\Theta_0$. If there is no such point, this means $\Gamma_0$ has no pair of antipodal points and thus complactly included in an open hemisphere from Remark \ref{antipodal_pts}. Then $|\Gamma_0|_n$ has to be strictly less than $|\mathbb{S}^n|_n$ as $\Gamma_0$ is outer area minimizing and contained in an geodesic ball whose radius is strictly less than $ \pi/2$. This is a contradiction. We can actually apply the same argument to $\Theta_0^{n-1} \subset \mathbb{S}^n$ and repeat this until the end. 
\end{proof}
\end{lemma} 

\begin{proof}[Proof of Theorem \ref{thm-sph}]
We again assume $p=e_1$ without loss of generality. The inequality $T_{\Theta_0}\leq T_{\Gamma_0}$ and the equality case follow directly from the previous lemma. Denote by ${\Gamma'}_t:=\left(\mathbb{R}\times \mathcal{C}\Theta_t\right)\cap \mathbb{S}^{n+1}$. Then we deduce $\hat{\Gamma}_0\subset \hat{\Gamma}'_0$ from the convexity of $\Gamma_0$. Furthermore, $\Gamma'_t$ is a IMCF. Therefore $T_p\hat{\Gamma}^n_t\cap \mathbb{S}^n_p$ is contained in $\hat{\Theta}_t$ from the avoidance principle and thus $p\in \Gamma_t$ for $t\in [0,T_{\Theta_0}]$. To get the other direction of inclusion, let $\Sigma^n_0\subset \mathbb{R}^{n+2}$ be a closed convex hypersurface which contains the origin and satisfies
\begin{align*}
B^{n+2}_2 (0)\cap\Sigma_0=B^{n+2}_2 (0)\cap\mathcal{C}\Gamma_0.
\end{align*}
Denote by $\Sigma_t$ the IMCF starting from $\Sigma_0$. From Theorem \ref{thm-euc} we have for any $t\in [0, T_{\Theta_0})$, $T_0\Sigma_t=\mathcal{C}\Gamma_t$ and $T_p\Sigma_t=\mathcal{C}\Gamma'_t$. Therefore by the convexity of $\Sigma_t$, $\hat{\Theta}_t=T_p\hat{\Sigma}_t\cap \mathbb{S}^n_p \subset T_p \mathcal{C}\hat{\Gamma}_t\cap \mathbb{S}^n_p=T_p\hat{\Gamma}_t\cap \mathbb{S}^n_p.$ The proof of $p\in \text{int}\hat{\Gamma}_t$ for $t\in (T_{\Theta_0},T_{\Gamma_0})$ is similar to the one for Theorem \ref{thm-euc}.
\end{proof}

\begin{remark}[Non-compact flow] \label{non-compact} As we intended, the result shown up to this point applies to non-compact solutions as well. Here we give a description of non-compact flow in $\mathbb{R}^{n+2}$ by collecting results from \cite{CD}. 
 
 Suppose $\Sigma^{n+1}_0=\partial\hat \Sigma_0$ is complete non-compact convex hypersurface in $\mathbb{R}^{n+2}$ and assume $0\in \hat\Sigma_0$ for a simplicity. Then we have a unique tangent cone  at infinity, namely $\mathcal{C}\Gamma_0$, which is obtained by taking a blow-down: $\mathcal C \hat \Gamma_0= \cap_{\epsilon>0} \epsilon\hat \Sigma_0$ and $\mathcal{C}\Gamma_0 = \partial (\mathcal{C}\hat \Gamma_0)$. Similarly, let us denote $\mathcal{C}\Gamma_t=\partial(\mathcal{C} \hat \Gamma_t)$ by the tangent cone at infinity of $\Sigma_t$. It should be noted that $\Gamma_t$ may fail to be a convex `hypersurface' as $\hat \Gamma_t$ can be an arbitrary convex set in $\mathbb{S}^{n+1}$.

 In \cite{CD}, the first author and P. Daskalopoulos shows an existence of classical solution under an additional assumption that $\Sigma_0$ is locally $C^{1,1}$.  The construction of solution in \cite{CD} is the same as Definition \ref{weak solution} and hence our result gives a refinement of Choi-Daskalopoulos's result.  The following result from \cite{CD} also holds without $C^{1,1}_{loc}$ assumption as its proof directly applies to the weak solution without any changes: there is an existence of non-empty weak solution $\Sigma_t= \partial \hat \Sigma_t$ up to time $t<T^*= \ln |\mathbb{S}^n| - \ln P(\hat \Gamma_0) \in [0,\infty]$ where \[\begin{aligned} P(\hat \Gamma_0) :&= \text{the perimeter of  }\hat\Gamma_0\text{ in }\mathbb{S}^{n+1}\\ &=\begin{cases} \begin{aligned} &|\Gamma_0| &&\text{if } \text{int} (\hat\Gamma_0) \text{ is non-empty and hence }\Gamma_0\text{ is  a hypersurface}\\ 2&|\Gamma_0|&&\text{otherwise}.\end{aligned}\end{cases} \end{aligned}\] $T^*$ is maximal time in the sense that, for $t>T^*$, $\hat \Sigma_t =\mathbb{R}^{n+2}$ and thus $\Sigma_t$ is empty. If $|\Gamma_0|>0$, it was also shown that $\Gamma_t$ becomes a convex hypersurface for $t>0$. 
\end{remark}
 
Next, though it was obviously expected, Choi-Daskalopoulos \cite{CD} could not conclude $\Gamma_t$ evolves by IMCF in $\mathbb{S}^{n+1}$ as there was no notion of weak solution. We can now prove this assertion.
\begin{claim}\label{blow-down}
The blow-down cone at infinity $\Gamma_t$ evolves by the weak IMCF for $t\in(0,T^*)$.
\end{claim}

\begin{proof}[Proof of Claim] 
Let $t_0 \in (0,T^*)$ be a given time. It suffices to show the claim for $t\in[t_0,T^*)$. Suppose $\Gamma_{t_0}$ has a pair of antipodal points, then $\Sigma_{t_0}$ contains an infinite line and this implies $\Sigma_{t_0}$ splits in the direction of the line. It is not hard to check that the weak solution $\Gamma_t$ also splits in the same direction for $t>t_0$. We can now apply dimension reduction and we may assume  $\Gamma_{t_0}$ has no pair of antipodal points in $\mathbb{S}^{n+1}$. Therefore $\Gamma_{t_0}$ is contained in an open hemisphere by Remark \ref{antipodal_pts}.

Next, we can find two families of convex hypersurfaces in $\mathbb{S}^{n+1}$ which correspond to monotone smooth strictly convex approximations from inside and outside. We denote them by $\{\Gamma_{t_0,\epsilon} \}_{\epsilon<0}$ and $\{\Gamma_{t_0,\epsilon} \}_{\epsilon>0}$, respectively. To be precise, we have $\{\Gamma_{t_0,\epsilon}\}_{\epsilon\in(-\delta,0)\cup(0,\delta)}$ with $\partial \hat \Gamma_{t_0,\epsilon} = \Gamma_{t_0,\epsilon}$ having the following properties: \begin{enumerate}\item $\hat\Gamma_{t_0,\epsilon_1} \subset \subset \hat\Gamma_{t_0,\epsilon_2}$ if $\epsilon_1<\epsilon_2$ 
 \item There is a hemisphere which contains $\Gamma_{t_0,\epsilon}$ for all $\epsilon \in (-\delta,0)\cup(0,\delta)$.
 \item Each $\Gamma_{t_0,\epsilon}$ is smooth strictly convex hypersurface in $\mathbb{S}^{n+1}$.
 \item $\overline{\cup_{\epsilon<0} \hat \Gamma_{t_0,\epsilon} } =\hat \Gamma_{t_0}$  and $\cap _{\epsilon>0} \hat \Gamma_{t_0,\epsilon}  =\hat \Gamma_{t_0}$. 
 \end{enumerate} For such approximations, we have $|\Gamma_{t_0,h}| \uparrow |\Gamma_{t_0}|$ as $h\uparrow 0$ and $|\Gamma_{t_0,h}| \downarrow |\Gamma_{t_0}|$ as $h\downarrow 0$.  Such an approximation from inside could be obtained by using the mean curvature flow and an approximation from outside is constructed in Claim 4.1 of \cite{CD}. Let us denote $\Gamma_{t,\epsilon}$ for $t\ge t_0$ by the IMCF running from $\Gamma_{t_0,\epsilon}$. Our goal is to show $\hat \Gamma_t = \overline{\cup_{\epsilon<0} \hat \Gamma_{t,\epsilon}}$. 
 
For $h\in(0,\delta)$, $\mathcal{C} \hat \Gamma_{t_0,-h} \subset \hat\Sigma_{t_0}$ by convexity. We can check that, for fixed $h$, $\mathcal{C}\Gamma_{t,-h}\subset \mathbb{R}^{n+2}$, $t\ge t_0$, is the weak solution running from $\mathcal{C}\Gamma_{t_0,-h}$. The solution from $\mathcal{C}\Gamma_{t_0,-h}$, say $N_t$, has to be a cone as the scaling invariance property should be preserved for the weak solutions and Theorem \ref{thm-euc} implies $N_t$ has to coincide with $\mathcal{C}\Gamma_{t,-h}$. The avoidance principle between weak solutions implies $\mathcal{C}\hat \Gamma_{t,-h} \subset \hat \Sigma_{t}$ for $t\ge t_0$.  By taking a blow-down at time $t$, this proves $\hat \Gamma_{t,-h} \subset \hat\Gamma_t$ for $t\ge t_0$.
On the other hand, for $h\in(0,\delta)$  we may find some $v\in\mathbb{R}^{n+2}$ such that $\hat \Sigma_{t_0} \subset \mathcal{C} \hat \Gamma_{t_0,h} +v$. By the avoidance principle, $\hat \Sigma_{t} \subset \mathcal{C} \hat \Gamma_{t,h} +v$ and this implies $\hat \Gamma_t \subset\hat \Gamma_{t,h}$. In summary, \[\overline{\cup_{\epsilon<0} \hat\Gamma_{t,\epsilon}} \subset\hat \Gamma_{t} \subset \cap_{\epsilon>0} \hat\Gamma_{t,\epsilon}.\]
In addition, \[|\partial(\cup_{\epsilon<0} \hat\Gamma_{t,\epsilon})|=\lim_{\epsilon\to 0-}|\Gamma_{t,\epsilon}|  =e^{t-t_0}\lim_{\epsilon\to 0-}|\Gamma_{t_0,\epsilon}|=|\Gamma_{t_0}|, \]
\[|\partial(\cap_{\epsilon>0} \hat\Gamma_{t,\epsilon})|=\lim_{\epsilon\to 0+}|\Gamma_{t,\epsilon}|  =e^{t-t_0}\lim_{\epsilon\to 0+}|\Gamma_{t_0,\epsilon}|=|\Gamma_{t_0}|, \] and the strict outer area minimizing property of a closed convex set in an open hemisphere imply \[\overline{\cup_{\epsilon<0} \hat\Gamma_{t,\epsilon}}=\hat \Gamma_{t} = \cap_{\epsilon>0} \hat\Gamma_{t,\epsilon}.\]
 
\end{proof} 

Now we are ready to describe the smoothing time of complete convex IMCF in $\mathbb{R}^{n+2}$ and in $\mathbb{S}^{n+1}$.

\begin{corollary}\label{cor-smoothingtime}
Let $\Sigma^{n+1}_0$ be a  complete convex hypersurface in $\mathbb{R}^{n+2}$ and $\Sigma_t$ be the IMCF starting from $\Sigma_0$. Let $T^*$ be the maximal time defined in Remark \ref{non-compact}. Then the solution becomes smooth for $t>T=T(\Sigma_0)$ with \[T:= \sup \{ \ln |\mathbb{S}^{n}| -\ln |\Gamma^n_0(p)|\, :\, \Gamma^n_0(p) = T_p\Sigma_0 \cap \partial B_1(p),\, p\in \Sigma_0\}.\] provided $T<T^*$
Moreover, if $\Sigma_0$ is compact, then there is a bound $T<T_0=(n+1)\left(\ln L_0-\ln r_0\right)$ where $L_0$ is the diameter of $\Sigma_0$ and $r_0$ is the largest radius of a ball contained in $\hat \Sigma_0$. \\

%
Similarly, let $\Gamma^n_t\subset \mathbb{S}^{n+1}$ be the solution of IMCF starting from a convex compact hypersurface $\Gamma^n_0\subset\mathbb{S}^{n+1}$. Then the solution becomes smooth for $t>T=T(\Gamma_0)$ with \[T:= \sup \{ \ln |\mathbb{S}^{n-1}| -\ln |\Theta^{n-1}_0(p)|\, :\, \Theta^{n-1}_0(p) = T_p\Gamma_0 \cap \mathbb{S}^n_p,\, p\in \Gamma_0\}\] provided $T<T_{\Gamma_0}$.
\begin{proof} 
First, we consider $\Sigma_t \subset \mathbb{R}^{n+2}$. Without loss of generality we can assume $\Sigma_0$ doesn't split a line and hence the link of the blow-down cone is compactly contained in some open hemisphere. We fix $T<t_0<t_1<T^*$. For $t\ge t_0$, $\hat \Sigma_0\subset \text{int} \hat \Sigma_{t_0}\subset \text{int} \hat \Sigma_{t}$. In view of Proposition 2.11 \cite{DH}, we have a local upper bound of $H$ for $t\ge t_0$. By convexity, this implies the second fundamental form is locally bounded.\ {In addition to this, for $t\leq t_1$, a local lower bound of $H$ is given by Theorem \ref{thm-CD} together with the outer barrier constructed as in the proof of Claim \ref{blow-down}.}\ The solution can be locally written as a graph over a small disk. If $x'$ is a coordinate of the disk and $u(x,'t)$ is the height function of the graph, it solves \[u_t = -(1+|Du|^2)^{1/2} \left[ \text{div} \frac{Du}{(1+|Du|^2)^{1/2}} \right] ^{-1}\] and this is uniformly parabolic if $|Du|$, $H$, and $H^{-1}$ are bounded.  We can always make $|Du|$ bounded by using the bound on the second fundamental form. Then higher regularity theory implies the solution is smooth in $(t_0,t_1]$.  \\


Next, we consider $\Gamma_t\subset \mathbb{S}^{n+1}$. From the assumption $T<T_{\Gamma_0}$ $\hat{\Gamma}_0$ contains no antipodal points and $\Gamma_0$ is compactly contained in some open hemisphere. Denote by $\Sigma_0=\mathcal{C}\Gamma_0\subset \mathbb{R}^{n+2}$ the cone over $\Gamma_0$. Since $\Sigma_t=\mathcal{C}\Gamma_t$, the result follows by applying the same argument as above in the region $B^{n+2}_2(0)\setminus B^{n+2}_1(0)$.
\end{proof}

\end{corollary}

\begin{remark}\label{finremark}Note that
\begin{enumerate}

\item For the smoothing time $T$ and the maximal time of existence $T^*$ of flows in $\mathbb{R}^{n+2}$, it could be easily verified that $T^*\ge T$ and the equality holds if and only if $\Sigma_0$ is a cone.

 \item For the flow in $\mathbb{R}^{n+2}$,  $$\frac{|\Gamma^n_0(p)|}{|\mathbb{S}^n|}= \lim_{r\to0} \frac{|B_r(p) \cap \Sigma_0|}{\omega_{n+1} r^{n+1}}=:\rho(p)$$ which is the density of the hypersurface $\Sigma_0^{n+1}\subset \mathbb{R}^{n+2}$ at $p$. The same thing can be written for $\Gamma_t^n \subset \mathbb{S}^{n+1}$. Hence, we have an exact formula of time $T$: 
 \[T(\Sigma_0)=\sup\,\{-\ln\rho(p)\,:\, p \in \Sigma_0\}\quad\text{and}\quad T(\Gamma_0) =\sup\,\{-\ln\rho(p)\,:\, p \in \Gamma_0\}.\]

\item When initial hypersurface is compact, the supremum in the definition of $T$ is actually achieved by a point in $\Sigma_0$. This is because the density $|\Gamma_0(p)|$ is a lower semi-continuous function of $p$ {due to Bishop-Gromov volume comparison \cite{BBI}.} The same happens for $\Gamma_0$ and $|\Theta_0(p)|$.

\item Corolloary \ref{cor-smoothingtime} implies our weak solution becomes smooth for $t>0$ if $\rho\equiv 1$. On the other hand, if there is a point $q$ with $\rho(q)<1$, by locating conical barriers on outside (like one we did in the proof of Lemma \ref{main lemma}), we can check there is no smooth solution upto $t<-\ln\rho(q)$. Therefore $\rho\equiv 1$ is a {\em sufficient and necessary condition} for an existence of smooth classical solution.  Note if $\rho(q)=1$, the tangent cone at $q$ has to be either a plane or an wedge by Lemma \ref{wedge}. It is interesting to note that for hypersurfaces, $$\{\text{convex and }C^1\}\subsetneq\{\text{convex and } \rho\equiv1 \} \subsetneq \{\text{convex}\}.$$

 \end{enumerate}\end{remark}

 This result gives a concrete picture of singular IMCF for convex hypersurface, and in particular, this answers our original question on a cube.
\begin{example}[Cube in $\mathbb{R}^3$]\label{example-cube} \begin{figure}
  \centering{\tiny
  \def\svgwidth{\columnwidth}
    \resizebox{0.7\textwidth}{!}{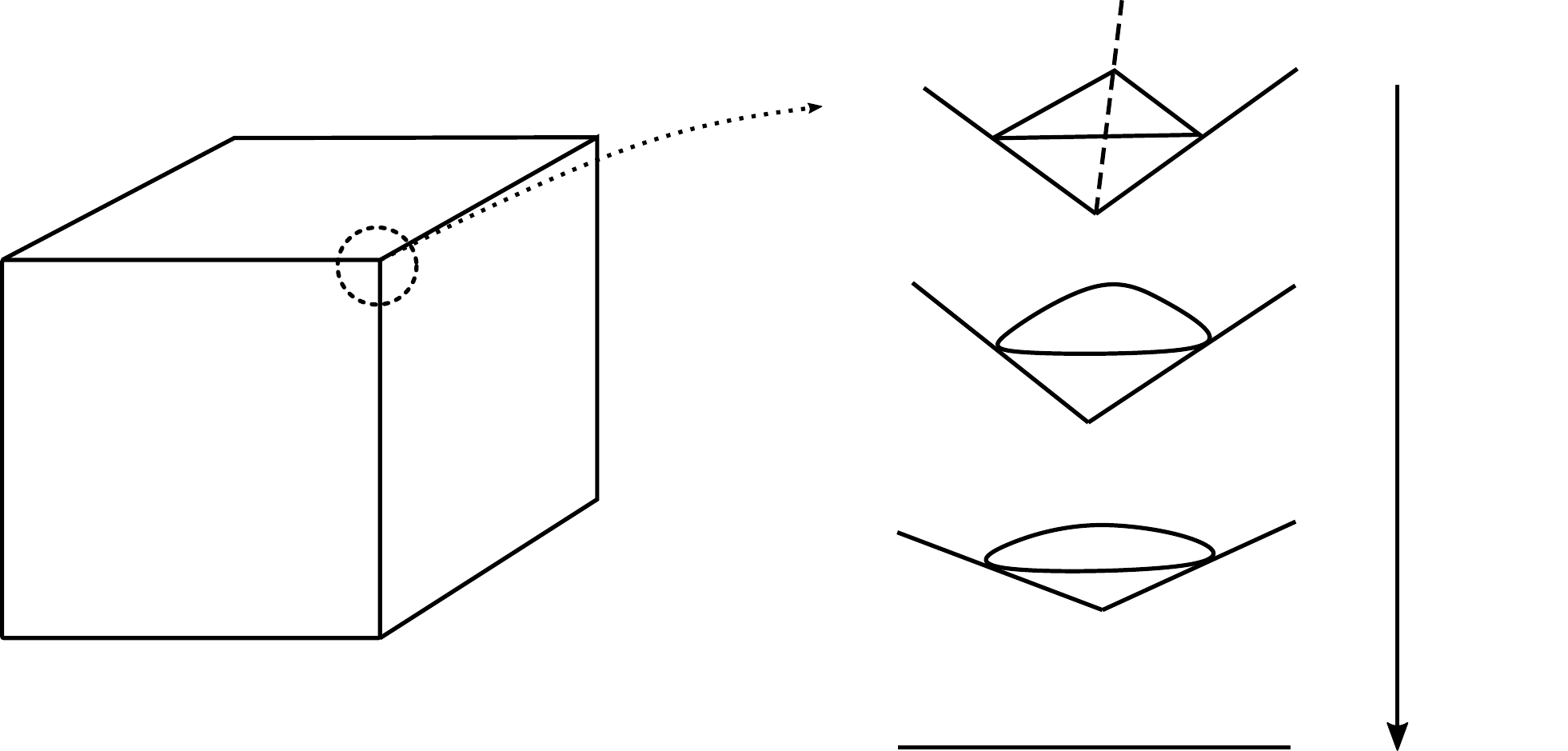}
  \caption{Example \ref{example-cube}}}
\end{figure}Let $\Sigma_0^2$ be a cube in $\mathbb{R}^3$ with the equal length of edges. 
 The IMCF $\Sigma_t$ has the following properties: 

\begin{enumerate} 
\item Edges are smoothened for $t>0$ by Theorem \ref{thm-euc} and the solution is smooth away from the corners. 
\item Each corner stays their original position until $t=\ln 2\pi - \ln \frac{3}{2}\pi = \ln 4/3$ and it evolves by IMCF by Theorem \ref{thm-euc}.
\item The solution becomes strictly convex for $t>0$ by strict convexity theorem in Appendix of \cite{CD}
\item The solution becomes smooth for $t> \ln 4/3$. Also, it expands and converges to an expanding sphere as $t\to\infty$ by \cite{Ur} and \cite{Ge2}. 
\end{enumerate}

\end{example}

\end{document}